\documentclass[12pt]{amsart}
\usepackage{amssymb}
\usepackage{enumerate}
\usepackage{graphicx}

\makeatletter
\@namedef{subjclassname@2010}{%
  \textup{2010} Mathematics Subject Classification}
\makeatother
\newtheorem{thm}{Theorem}[section]
\newtheorem{cor}[thm]{Corollary}
\newtheorem{lem}[thm]{Lemma}
\newtheorem{prop}[thm]{Proposition}

\theoremstyle{definition}
\newtheorem{defin}[thm]{Definition}
\newtheorem{rem}[thm]{Remark}

\numberwithin{equation}{section}

\newcommand{\Hom}{{\rm Hom}}
\newcommand{\Ext}{{\rm Ext}}
\newcommand{\HH}{{\rm HH}}
\newcommand{\rad}{{\rm rad\,}}
\newcommand{\Ker}{{\rm Ker\,}}
\newcommand{\Ima}{{\rm Im\,}}

\begin{document}

\title[Hochschild cohomology of a self-injective algebra]
{On Hochschild cohomology of a self-injective special biserial algebra 
obtained by a circular quiver with double arrows}

\author[A. Itaba]
{Ayako Itaba}
\address{
Ayako Itaba\\
Department of Mathematics, Tokyo University of Science \\
Kagurazaka 1-3, Shinjuku, Tokyo 162-0827, Japan}
\email{j1110701@ed.tus.ac.jp}

\date{\today}

\begin{abstract}
  We calculate the dimensions of the Hochschild cohomology groups of 
  a self-injective special biserial algebra $\Lambda_{s}$ 
  obtained by a circular quiver with double arrows. 
  Moreover, we give a presentation 
  of the Hochschild cohomology ring modulo nilpotence 
  of $\Lambda_{s}$ by generators and relations. 
  This result shows that the Hochschild cohomology ring 
  modulo nilpotence of $\Lambda_{s}$ is finitely generated as an algebra. 
\end{abstract}



\maketitle
\section{Introduction}
Let $K$ be an algebraically closed field. 
For a positive integer $s$, 
let $\Gamma_{s}$ be the following circular quiver with double arrows: 
 \begin{center}
  \includegraphics[scale=0.9]{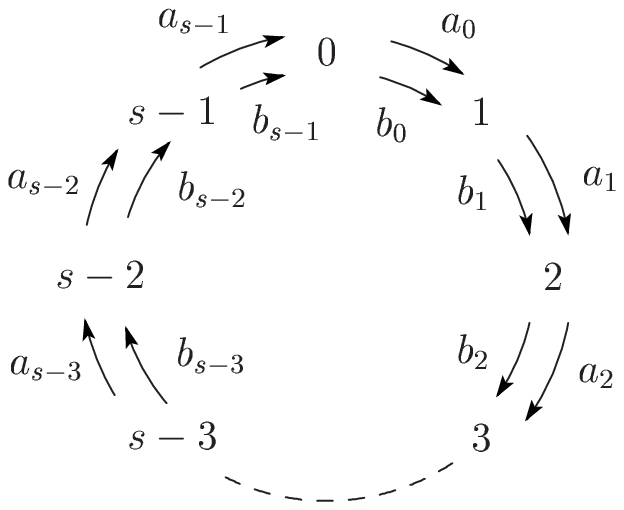}
 \end{center}
Denote the trivial path at the vertex $i$ by $e_{i}$. 
We set the elements $x=\sum_{i=0}^{s-1} a_{i}$ 
and $y=\sum_{i=0}^{s-1} b_{i}$ in the path algebra $K\Gamma_{s}$. 
Then $e_{i}x^{n}=x^{n}e_{i+n}=e_{i}x^{n}e_{i+n}$ 
and $e_{i}y^{n}=y^{n}e_{i+n}=e_{i}y^{n}e_{i+n}$ hold 
for $0\leq i \leq s-1$ and $n\geq 0$, 
where the subscript $i+n$ of $e_{i+n}$ is regarded as modulo $s$. 
We denote by $I$ the ideal generated by $x^{2}$, $xy+yx$ and $y^{2}$, 
that is, 
$I=\langle x^{2},xy+yx,y^{2}\rangle=\langle e_{i}x^{2},e_{i}(xy+yx),e_{i}y^{2}\ |\ 0\leq i \leq s-1 \rangle$. 
Then we define the bound quiver algebra $\Lambda_{s}=K\Gamma_{s}/I$ over $K$. 
This algebra $\Lambda_{s}$ is a Koszul self-injective special biserial algebra 
for $s\geq 1$ (see Proposition 2.2), 
but is not a weakly symmetric algebra for $s\geq 3$. 
Our purpose in this paper is to study the Hochschild cohomology 
of $\Lambda_{s}$ for $s\geq 3$. 

We immediately see that $\Lambda_{1}$ is the exterior algebra 
in two variables $K[x,y]/\langle x^{2},xy+yx,y^{2}\rangle$, 
and Xu and Han have studied the Hochschild cohomology 
groups and rings of the exterior algebras 
in arbitrary variables in \cite{XH}. 
Also, in \cite{ST} and \cite{F},
the Hochschild cohomology groups and rings for classes of 
some self-injective special biserial algebras have been studied. 
We notice that these classes contain algebras 
isomorphic to $\Lambda_{2}$ and $\Lambda_{4}$. 

In \cite{GHMS}, 
Green, Hartmann, Marcos and Solberg constructed 
a minimal projective bimodule resolution for any Koszul algebra 
by using sets $\mathcal{G}^{n}$ $(n\geq 0)$ introduced 
by Green, Solberg and Zacharia in \cite{GSZ}. 
Moreover, by using this method, 
minimal projective bimodule resolutions of several weakly symmetric algebras 
are constructed in \cite{FO}, \cite{ST} and \cite{ScSn}. 
In this paper, by the same method, 
we give a minimal projective bimodule resolution of $\Lambda_{s}$ 
for $s\geq 1$ and compute the Hochschild cohomology group 
$\HH^{n}(\Lambda_{s})$ of $\Lambda_{s}$ for $n\geq 0$ 
in the case where $s\geq 3$. 

In \cite{SnSo}, Snashall and Solberg have defined 
the support varieties of finitely generated modules 
over a finite-dimensional algebra 
by using the Hochschild cohomology ring modulo nilpotence, 
which are analogous to the support varieties 
for group algebras of finite groups. 
Furthermore, in \cite{SnSo}, Snashall and Solberg have conjectured 
that the Hochschild cohomology ring modulo nilpotence 
is finitely generated as an algebra. 
So far, it has been proved that 
the Hochschild cohomology rings modulo nilpotence of 
the following classes of finite-dimensional algebras 
are finitely generated as algebras: 
group algebras of finite groups (\cite{E}); 
self-injective algebras of finite representation type (\cite{GSS1}); 
monomial algebras (\cite{GSS2}); 
several self-injective special biserial algebras 
(\cite{ES}, \cite{F}, \cite{ScSn}, \cite{ST}). 
However, in \cite{X}, Xu has found 
a counterexample to this conjecture 
(see also \cite{S}, \cite{XZ}). 
In this paper, we give generators and relations 
of the Hochschild cohomology ring modulo nilpotence 
$\HH^{\ast}(\Lambda_{s})/\mathcal{N}_{\Lambda_{s}}$ 
of $\Lambda_{s}$ for all $s\geq 3$, 
and show that $\HH^{\ast}(\Lambda_{s})/\mathcal{N}_{\Lambda_{s}}$ 
is finitely generated as an algebra. 

This paper is organized as follows: 
In Section 2, we give sets $\mathcal{G}^{n}$ $(n\geq 0)$ 
for the right $\Lambda_{s}$-module $\Lambda_{s}/\rad \Lambda_{s}$. 
Moreover, by using $\mathcal{G}^{n}$, 
we construct a minimal projective resolution of $\Lambda_{s}$ 
as a $\Lambda_{s}$-$\Lambda_{s}$-bimodule (see Theorem 2.6). 
In Section 3, we find a $K$-basis of the Hochschild cohomology group 
$\HH^{n}(\Lambda_{s})$ (see Proposition 3.8) 
and describe the dimension of $\HH^{n}(\Lambda_{s})$ 
for $n\geq 0$ and $s\geq 3$ (see Theorem 3.9). 
In Section 4, we investigate the Hochschild cohomology ring modulo nilpotence 
$\HH^{\ast}(\Lambda_{s})/\mathcal{N}_{\Lambda_{s}}$ for $s\geq 3$. 
In particular, it is shown that $\HH^{\ast}(\Lambda_{s})/\mathcal{N}_{\Lambda_{s}}$ 
is finitely generated as an algebra for $s\geq 3$ (see Theorem 4.1). 

Throughout this paper, for all arrows $a$ of $\Gamma_{s}$, 
we denote the origin of $a$ by $o(a)$ and the terminus of $a$ by $t(a)$. 
Also, we denote the enveloping algebra 
$\Lambda_{s}^{\text{op}}\otimes_{K} \Lambda_{s}$ 
of $\Lambda_s$ by $\Lambda_{s}^{e}$. 
Note that there is a natural one to one correspondence 
between the family of $\Lambda_{s}$-$\Lambda_{s}$-bimodules 
and that of right $\Lambda_{s}^{e}$-modules. 
For the general notation, 
we refer to \cite{ASS}. 
\section{Sets $\mathcal{G}^{n}$ for $\Lambda_{s}/\rad\Lambda_{s}$, 
         and a projective bimodule resolution $(Q^{\bullet},\partial^{\bullet})$ 
         of $\Lambda_{s}$}
Let $A=KQ/I$ be a finite-dimensional $K$-algebra, 
where $Q$ is a finite quiver and $I$ is an admissible ideal of $KQ$. 
We start by recalling the construction of sets 
$\mathcal{G}^{n}$ $(n\geq 0)$ in \cite{GSZ}. 
Let $\mathcal{G}^{0}$ be the set of vertices of $Q$, 
$\mathcal{G}^{1}$ the set of arrows of $Q$, 
and $\mathcal{G}^{2}$ a minimal set of uniform generators of $I$. 
In \cite{GSZ}, Green, Solberg and Zacharia proved that, for each $n\geq 3$, 
we have a set $\mathcal{G}^{n}$ consisting of uniform elements in $KQ$ 
such that there is a minimal projective resolution $(P^{\bullet},d^{\bullet})$ 
of the right $A$-module $A/\rad A$ satisfying 
the following conditions (a), (b) and (c): 
 \begin{enumerate}[(a)]
  \item For each $n\geq 0$, $P^{n}=\bigoplus_{g\in \mathcal{G}^{n}}t(x)A$. 
  \item For each $g\in \mathcal{G}^{n}$, 
        we have unique elements $r_{h},s_{k}\in KQ$, 
        where $h\in \mathcal{G}^{n-1}$ and $k\in \mathcal{G}^{n-2}$, 
        satisfying 
        $g=\sum_{h\in \mathcal{G}^{n-1}}hr_{h}=\sum_{k\in \mathcal{G}^{n-2}}ks_{k}$. 
  \item For each $n \geq 1$, $d^{n}:P^{n}\rightarrow P^{n-1}$ 
        is determined by 
        $d^{n}(t(g)\lambda)=\sum_{h\in \mathcal{G}^{n-1}}r_{h}t(g)\lambda$ 
        for $g\in \mathcal{G}^{n}$ and $\lambda\in A$, 
        where $r_{h}$ denotes the element in (b). 
 \end{enumerate}

In \cite{GHMS}, a minimal projective bimodule resolution of any Koszul algebra is given 
by using the above sets $\mathcal{G}^{n}$ $(n\geq 0)$. 
In this section, we construct sets $\mathcal{G}^{n}$ $(n\geq 0)$ 
for the right $\Lambda_{s}$-modules $\Lambda_{s}/\rad \Lambda_{s}$, 
and then we give a minimal projective bimodule resolution 
$(Q^{\bullet},\partial^{\bullet})$ of $\Lambda_{s}$ by following \cite{GHMS}. 
\subsection{Sets $\mathcal{G}^{n}$ for $\Lambda_{s}/\rad\Lambda_{s}$}
In order to construct sets $\mathcal{G}^{n}$ for $\Lambda_{s}/\rad \Lambda_{s}$, 
we define the following elements $g_{i,j}^{n}$ in $K\Gamma_{s}$: 
 \begin{defin}
  For $0\leq i \leq s-1$, we put $g_{i,0}^{0}:=e_{i}$, and, 
  for $n\geq 1$, we inductively define the elements 
  $g_{i,j}^{n}\in K\Gamma_{s}$ as follows{\rm :} 
   \begin{itemize}
    \item $g_{i,0}^{n}:= g_{i,0}^{n-1}y$\, for $0 \leq i \leq s-1$, 
    \item $g_{i,j}^{n}:= g_{i,j-1}^{n-1}x + g_{i,j}^{n-1}y$\, 
          for $0 \leq i \leq s-1$ \text{ and } $1\leq j \leq n-1$, 
    \item $g_{i,n}^{n}:= g_{i,n-1}^{n-1}x$ \, for $0 \leq i \leq s-1$, 
   \end{itemize}
 \end{defin}
Throughout this paper, 
we regard the subscript $i$ 
of $g_{i,\bullet}^{\bullet}$ as modulo $s$. 
Then we see that these elements $g_{i,j}^{n}$ are uniform, 
and that $o(g_{i,j}^{n})=e_{i}$ and $t(g_{i,j}^{n})=e_{i+n}$ 
hold for all $n\geq 0$, $i\in \mathbb{Z}$, and $0\leq j\leq n$. 

We put the set 
 \[
  \mathcal{G}^{n}=\{ g_{i,j}^{n} \mid 0 \leq i \leq s-1,\, 0\leq j \leq n\}
  \]
for all $n \geq 0$. 
It is easy to check that these sets satisfy the conditions 
(a), (b) and (c) in the beginning of this section. 

Now, we have the following proposition. 
 \begin{prop}
  For $s\geq 1$, the algebra $\Lambda_{s}$ is 
  a Koszul self-injective special biserial algebra. 
 \end{prop}
 \begin{proof}
  By calculating directly 
  the presentations of all indecomposable projective 
  and injective $\Lambda_{s}$-modules, 
  we easily see that $\Lambda_{s}$ is self-injective. 
  Moreover, let $d^{0}:P^{0}:=\bigoplus_{x\in \mathcal{G}^{0}}t(x)\Lambda_{s}\rightarrow \Lambda_{s}/\rad \Lambda_{s}$ 
  be the natural map. 
  Then the projective resolution $(P^{\bullet},d^{\bullet})$ 
  given by (a), (b) and (c) is a linear resolution of $\Lambda_{s}/\rad \Lambda_{s}$. 
  Therefore, $\Lambda_{s}$ is a Koszul algebra. 
 \end{proof}
In order to obtain a minimal projective resolution $(Q^{\bullet},\partial^{\bullet})$ 
of $\Lambda_{s}$ as a $\Lambda_{s}^{e}$-module, 
we need the following lemma. 
 \begin{lem}
  For $n\geq 1$, the following equations hold{\rm :} 
   \begin{itemize}
    \item $g_{i,0}^{n} = yg_{i+1,0}^{n-1}$ \quad for $0 \leq i \leq s-1$, 
    \item $g_{i,j}^{n} = yg_{i+1,j}^{n-1} + xg_{i+1,j-1}^{n-1}$ 
          \quad for $0 \leq i \leq s-1$\text{ and } $1\leq j \leq n-1$, 
    \item $g_{i,n}^{n} = xg_{i+1,n-1}^{n-1}$ \quad for $0 \leq i \leq s-1$. 
   \end{itemize}
 \end{lem}
\noindent
The proof is easily done by induction on $n$. 
\subsection{A minimal projective resolution 
           of $\Lambda_{s}$ as a right $\Lambda_{s}^{e}$-module}
In this subsection, 
by using the sets $\mathcal{G}^n$ ($n\geq 0$) of Section 2.1, 
we give a minimal projective resolution $(Q^{\bullet},\partial^{\bullet})$ of $\Lambda_{s}$ 
as a right $\Lambda_{s}^{e}$-module. 

First, we start with the definition of the projective 
$\Lambda_{s}^{e}$-module $Q^{n}$ for $n\geq 0$. 
For simplicity, we denote $\otimes_{K}$ by $\otimes$ and, 
for $n\geq 0$, set the elements 
$b_{i,j}^{n}:=o(g_{i,j}^{n})\otimes t(g_{i,j}^{n})$ 
in $\Lambda_{s}o(g_{i,j}^{n})\otimes t(g_{i,j}^{n})\Lambda_{s}$ 
for $0\leq i \leq s-1$ and $0\leq j \leq n$. 
 \begin{defin}
  We define the projective $\Lambda_{s}^e$-module $Q^{n}$ by 
   \[
    Q^{n}:=\bigoplus_{g \in \mathcal{G}^{n}}\Lambda_{s}o(g)\otimes t(g)\Lambda_{s}=\bigoplus_{i=0}^{s-1} \bigoplus_{j=0}^{n} \Lambda_{s} b_{i,j}^{n} \Lambda_{s} \quad\text{ for }n\geq 0.
   \]
 \end{defin}
Next, we define the map $\partial^{n}$: $Q^{n}\rightarrow Q^{n-1}$ 
as follows: 
 \begin{defin}
  We define $\partial^{0}:Q^{0}\rightarrow \Lambda_{s}$ to be the multiplication map, 
  and, for $n\geq 1$, $\partial^{n}:Q^{n}\rightarrow Q^{n-1}$ 
  to be the $\Lambda_{s}^{e}$-homomorphism determined by 
   \begin{itemize}
    \item $b_{i,0}^{n} \longmapsto b_{i,0}^{n-1}y+(-1)^{n}yb_{i+1,0}^{n-1}$ \quad for\, $0 \leq i \leq s-1$, 
    \item $b_{i,j}^{n} \longmapsto (b_{i,j-1}^{n-1}x+b_{i,j}^{n-1}y)+(-1)^{n}(yb_{i+1,j}^{n-1}+xb_{i+1,j-1}^{n-1})$ 
          \quad for\, $0 \leq i \leq s-1$ {\rm and} $1\leq j \leq n-1$,
    \item $b_{i,n}^{n} \longmapsto b_{i,n-1}^{n-1}x+(-1)^{n}xb_{i+1,n-1}^{n-1}$ \quad for \, $0 \leq i \leq s-1$, 
   \end{itemize}
  where the subscript $i+1$ of $b_{i+1,\bullet}^{\bullet}$ 
 is regarded as modulo $s$. 
 \end{defin}
\noindent
By direct computations, we see that the composite 
$\partial^{n}\partial^{n+1}$ is zero for all $n\geq 0$. 
Therefore, $(Q^{\bullet},\partial^{\bullet})$ is a complex of 
$\Lambda_{s}^{e}$-modules. 

Now, since $\Lambda_{s}$ is a Koszul algebra by Proposition 2.2, 
the following theorem is immediate from \cite{GHMS}. 
 \begin{thm}
  For $s\geq 1$, $(Q^{\bullet},\partial^{\bullet})$ is 
  a minimal projective $\Lambda_{s}^{e}$-resolution of $\Lambda_{s}$. 
 \end{thm}
\section{The Hochschild cohomology groups $\HH^{n}(\Lambda_{s})$} 
In this section, we calculate the Hochschild cohomology group 
$\HH^{n}(\Lambda_{s})$ for $n\geq 0$. 
By applying the functor $\Hom_{\Lambda_{s}^{e}}(-,\Lambda_{s})$ 
to the resolution $(Q^{\bullet},\partial^{\bullet})$, we have the complex 
 \[
  0 \longrightarrow \widehat{Q}^{0} \stackrel{\widehat{\partial}^{1}}{\longrightarrow} \widehat{Q}^{1} \stackrel{\widehat{\partial}^{2}}{\longrightarrow} \widehat{Q}^{2} \stackrel{\widehat{\partial}^{3}}{\longrightarrow} \cdots \stackrel{\widehat{\partial}^{n-1}}{\longrightarrow} \widehat{Q}^{n-1} \stackrel{\widehat{\partial}^{n}}{\longrightarrow} \widehat{Q}^{n} \stackrel{\widehat{\partial}^{n+1}}{\longrightarrow} \widehat{Q}^{n+1} \stackrel{\widehat{\partial}^{n+2}}{\longrightarrow} \cdots, 
 \]
where $\widehat{Q}^{n}:=\Hom_{\Lambda_{s}^{e}}(Q^{n},\Lambda_{s})$ 
and $\widehat{\partial}^{n}:=\Hom_{\Lambda_{s}^{e}}(\partial^{n},\Lambda_{s})$. 
We recall that, for $n\geq 0$, the $n$-th Hochschild cohomology group 
$\HH^{n}(\Lambda_{s})$ is defined to be the $K$-space 
$\HH^{n}(\Lambda_{s}):=\Ext_{\Lambda_{s}^{e}}^{n}(\Lambda_{s},\Lambda_{s})=\Ker\widehat{\partial}^{n+1}/\Ima\widehat{\partial}^{n}$. 
\subsection{A basis of $\widehat{Q}^{n}$}
We start with the following easy lemma. 
 \begin{lem}
  \label{lemma1}
   For integers $n\geq 0$, $0\leq i \leq s-1$ and $0\leq j \leq n$, 
   the $K$-space $o(g_{i,j}^{n})\Lambda_{s}t(g_{i,j}^{n})=e_{i}\Lambda_{s}e_{i+n}$ 
   has the following basis{\rm :} 
    $$
     \left\{
      \begin{array}{ll}
       1,x,y,xy &\quad \text{if } s=1, \\ 
       e_{i},e_{i}xy &\quad \text{if } s=2 \text{ and } n\equiv 0 \pmod{2},\\
       e_{i}x,e_{i}y &\quad  \text{if } s=2 \text{ and } n\equiv1 \pmod{2},\\
       e_{i} &\quad \text{if } s\geq3 \text{ and } n\equiv 0 \pmod{s},\\
       e_{i}x,e_{i}y &\quad  \text{if } s\geq 3 \text{ and } n\equiv 1 \pmod{s},\\
       e_{i}xy &\quad  \text{if } s\geq 3 \text{ and } n\equiv 2 \pmod{s}. 
      \end{array}
     \right.
    $$
   Also, if $s\geq 4$ and $n\not\equiv 0,1,2 \pmod{s}$, 
   then $o(g_{i,j}^{n})\Lambda_{s}t(g_{i,j}^{n})=0$. 
 \end{lem}
Let $n\geq 0$ be an integer. 
For $0\leq i\leq s-1$ and $0\leq j\leq n$, 
we define right $\Lambda_{s}^{e}$-module homomorphisms 
$\alpha_{i,j}^{n}$, $\beta_{i,j}^{n}$, $\gamma_{i,j}^{n}$, $\delta_{i,j}^{n}$
by the following equations: 
For $0\leq k \leq s-1$ and $0\leq l \leq n$, 
\begin{enumerate}[(i)]
 \item If $s=1$, we define $\alpha_{0,j}^{n}$, $\beta_{0,j}^{n}$, 
       $\gamma_{0,j}^{n}$, $\delta_{0,j}^{n}$: $Q^{n}\rightarrow \Lambda_{1}$ by 
        $$
         \begin{array}{llll}
          \alpha_{0,j}^{n}(b_{0,l}^{n})=
          \left\{
           \begin{array}{ll}
            e_{0}& \text{if }l=j,\\
            0    & \text{otherwise}, 
           \end{array}
          \right. 
          \beta_{0,j}^{n}(b_{0,l}^{n})=
          \left\{
           \begin{array}{ll}
            e_{0}x& \text{if }l=j,\\
            0     & \text{otherwise}, 
           \end{array}
          \right. \\[5mm]
          \gamma_{0,j}^{n}(b_{0,l}^{n})=
          \left\{
           \begin{array}{ll}
            e_{0}y& \text{if }l=j,\\
            0     & \text{otherwise}, 
           \end{array}
          \right. 
          \delta_{0,j}^{n}(b_{0,l}^{n})=
          \left\{
           \begin{array}{ll}
            e_{0}xy& \text{if }l=j,\\
            0      & \text{otherwise}. 
           \end{array}
          \right. 
         \end{array}
        $$\\
 \item If $s=2$, then 
       \begin{itemize}
        \item If $n\equiv 0$ $\pmod{2}$, 
             we define $\alpha_{i,j}^{n}$, 
             $\delta_{i,j}^{n}$: $Q^{n}\rightarrow \Lambda_{2}$ by 
             $$
              \begin{array}{llll}
              \alpha_{i,j}^{n}(b_{k,l}^{n})=
              \left\{
               \begin{array}{ll}
                e_{i}& \text{if }(k,l)=(i,j),\\
                0    & \text{otherwise}, 
               \end{array}
              \right. 
              \delta_{i,j}^{n}(b_{k,l}^{n})=
              \left\{
              \begin{array}{ll}
              e_{i}xy& \text{if }(k,l)=(i,j),\\
              0      & \text{otherwise}. 
              \end{array}
              \right. 
             \end{array}
            $$
      \item If $n\equiv 1$ $\pmod{2}$, 
            we define $\beta_{i,j}^{n}$, 
            $\gamma_{i,j}^{n}$: $Q^{n}\rightarrow \Lambda_{2}$ by 
             $$
              \begin{array}{llll}
              \beta_{i,j}^{n}(b_{k,l}^{n})=
              \left\{
              \begin{array}{ll}
               e_{i}x& \text{if }(k,l)=(i,j),\\
               0     & \text{otherwise}, 
               \end{array}
               \right. 
        	\gamma_{i,j}^{n}(b_{k,l}^{n})=
               \left\{
               \begin{array}{ll}
                e_{i}y& \text{if }(k,l)=(i,j),\\
                0     & \text{otherwise}. 
               \end{array}
              \right. 
              \end{array}
              $$
           \end{itemize}
 \item If $s\geq3$, then 
       \begin{itemize}
       \item If $n\equiv 0$ $\pmod{s}$, 
       we define $\alpha_{i,j}^{n}$: $Q^{n}\rightarrow \Lambda_{s}$ by 
        $$
         \begin{array}{llll}
          \alpha_{i,j}^{n}(b_{k,l}^{n})=
          \left\{
           \begin{array}{ll}
            e_{i}& \text{if }(k,l)=(i,j),\\
            0    & \text{otherwise}. 
           \end{array}
          \right. 
         \end{array}
        $$
       \item If $n\equiv 1$ $\pmod{s}$, 
       we define $\beta_{i,j}^{n}$, 
       $\gamma_{i,j}^{n}$: $Q^{n}\rightarrow \Lambda_{s}$ by 
        $$
         \begin{array}{llll}
          \beta_{i,j}^{n}(b_{k,l}^{n})=
          \left\{
           \begin{array}{ll}
            e_{i}x& \text{if }(k,l)=(i,j),\\
            0     & \text{otherwise}, 
           \end{array}
          \right. 
          \gamma_{i,j}^{n}(b_{k,l}^{n})=
          \left\{
           \begin{array}{ll}
            e_{i}y& \text{if }(k,l)=(i,j),\\
            0     & \text{otherwise}. 
           \end{array}
          \right. 
         \end{array}
        $$
       \item If $n\equiv 2$ $\pmod{s}$, 
       we define $\delta_{i,j}^{n}$: $Q^{n}\rightarrow \Lambda_{s}$ by
       $$
         \begin{array}{llll}
          \delta_{i,j}^{n}(b_{k,l}^{n})=
          \left\{
          \begin{array}{ll}
           e_{i}xy& \text{if }(k,l)=(i,j),\\
           0      & \text{otherwise}. 
          \end{array}
          \right. 
         \end{array}
        $$
        \end{itemize}
\end{enumerate}
Throughout this paper, 
we regard the subscripts $i$ 
of $\alpha_{i,\bullet}^{\bullet}$, $\beta_{i,\bullet}^{\bullet}$, $\gamma_{i,\bullet}^{\bullet}$ 
and $\delta_{i,\bullet}^{\bullet}$ as modulo $s$. 

For $n\geq 0$, 
we have an isomorphism of $K$-spaces 
$F_{n}$: $\bigoplus_{g\in \mathcal{G}^{n}}o(g)\Lambda_{s}t(g)\stackrel{\sim}{\rightarrow} \Hom_{\Lambda_{s}^{e}}(Q^{n},\Lambda_{s})$ 
given by $F_{n}(\sum_{g\in \mathcal{G}^{n}}z_{g})(b_{i,j}^{n})=z_{g_{i,j}^{n}}$, 
where $z_{g}\in o(g)\Lambda_{s}t(g)$ for $g\in \mathcal{G}^{n}$, 
$0\leq i\leq s-1$ and $0\leq j\leq n$ (cf. \cite{F}). 

Therefore, by the above isomorphism and Lemma~\ref{lemma1}, 
we have the following lemma: 
 \begin{lem}
  Let $n\geq 0$ be an integer. 
  Then the $K$-space $\widehat{Q}^{n}=\Hom_{\Lambda_{s}^{e}}(Q^{n},\Lambda_{s})$ 
  has the following basis{\rm :} 
   $$
    \left\{
     \begin{array}{ll}
      \alpha_{0,l}^{n},\beta_{0,l}^{n},\gamma_{0,l}^{n},\delta_{0,l}^{n}\ (0\leq l\leq n)    &\quad \text{if } s=1, \\
      \alpha_{k,l}^{n},\delta_{k,l}^{n} \ (k=0,1;\ 0\leq l \leq n)                             &\quad \text{if } s=2 \text{ and } n\equiv 0 \pmod{2},\\
      \beta_{k,l}^{n},\gamma_{k,l}^{n}\ (k=0,1;\ 0 \leq l \leq n)                               &\quad \text{if } s=2 \text{ and } n\equiv 1 \pmod{2},\\
      \alpha_{k,l}^{n} \ (0\leq k \leq s-1;\ 0\leq l \leq n)                                   &\quad \text{if } s\geq3 \text{ and } n\equiv 0 \pmod{s},\\
      \beta_{k,l}^{n},\gamma_{k,l}^{n} \ (0\leq k \leq s-1;\ 0\leq l \leq n)                   &\quad \text{if } s\geq3 \text{ and } n\equiv 1 \pmod{s},\\
      \delta_{k,l}^{n} \ (0\leq k \leq s-1;\ 0\leq l \leq n)                                   &\quad \text{if } s\geq3 \text{ and } n\equiv 2 \pmod{s}.
     \end{array}
    \right.
   $$
  Also, if $s\geq 4$ and $n \not\equiv 0,1,2 \pmod{s}$, then $\widehat{Q}^{n}=0$. 
 \end{lem}
\subsection{Maps $\widehat{\partial}^{n+1}$}
In the rest of this paper, 
we assume $s\geq 3$. 
By direct computations, 
we have the images of the basis elements in Lemma 3.2 
under the map 
$\widehat{\partial}^{n+1}=\Hom_{\Lambda_{s}^{e}}(\partial^{n+1},\Lambda_{s})$ for $n\geq 0$: 
 \begin{lem}
  For $0\leq i \leq s-1$ and $0 \leq j \leq n$, 
  we have the following equations{\rm :}
   \begin{enumerate}[\rm (i)]
    \item If $n\equiv 0 \pmod{s}$, then
                \begin{align*}
                 \widehat{\partial}^{n+1}(\alpha_{i,j}^{n})&=\alpha_{i,j}^{n}\partial^{n+1}\\
                                                             &=\beta_{i,j+1}^{n+1}+\gamma_{i,j}^{n+1}+(-1)^{n+1}(\beta_{i-1,j+1}^{n+1}+\gamma_{i-1,j}^{n+1}). 
                \end{align*}
    \item If $n\equiv 1 \pmod{s}$, then
                $$
                 \left\{
                  \begin{array}{lll}
                   \widehat{\partial}^{n+1}(\beta_{i,j}^{n})&=\beta_{i,j}^{n}\partial^{n+1}&=\delta_{i,j}^{n+1}+(-1)^n\delta_{i-1,j}^{n+1}, \\
                   \widehat{\partial}^{n+1}(\gamma_{i,j}^{n})&=\gamma_{i,j}^{n}\partial^{n+1}&=-\delta_{i,j+1}^{n+1}+(-1)^{n+1}\delta_{i-1,j+1}^{n+1}. 
                  \end{array}
                 \right.
                $$
    \item If $n\equiv 2 \pmod{s}$, then 
                 $$
                 \widehat{\partial}^{n+1}(\delta_{i,j}^{n})=\delta_{i,j}^{n}\partial^{n+1}=0. 
                 $$
   \end{enumerate}
  Moreover, if $n\not\equiv 0,1,2 \pmod{s}$, 
  then we have $\widehat{\partial}^{n+1}=0$. 
 \end{lem}
\subsection{A basis of $\Ima \widehat{\partial}^{n+1}$}
Now, by Lemma 3.3, we have the following lemma. 
 \begin{lem}
  Let $n\geq 0$ be any integer. 
  If we write $n=ms+r$ for integers $m$ and $0\leq r \leq s-1$, 
  then the following elements give a $K$-basis of the subspace 
  $\Ima \widehat{\partial}^{n+1}$ of $\widehat{Q}^{n+1}${\rm:} 
   \begin{enumerate}[\rm(a)]
    \item $\beta_{i,j+1}^{ms+1}+\gamma_{i,j}^{ms+1}-\beta_{i-1,j+1}^{ms+1}-\gamma_{i-1,j}^{ms+1}$ $(0\leq i\leq s-1$, $0\leq j \leq ms)$ 
          is a $K$-basis of $\Ima \widehat{\partial}^{ms+1}$ for $m$ odd, $s$ odd, and {\rm char}\,$K\neq 2$. 
    \item $\beta_{i,j+1}^{ms+1}+\gamma_{i,j}^{ms+1}-\beta_{i-1,j+1}^{ms+1}-\gamma_{i-1,j}^{ms+1}$ $(0\leq i\leq s-2$, $0\leq j \leq ms)$ 
          is a $K$-basis of $\Ima \widehat{\partial}^{ms+1}$ for $m$ even, $s$ even, or {\rm char}\,$K= 2$. 
    \item $\delta_{i,j}^{ms+2}+\delta_{i-1,j}^{ms+2}$ $(0\leq i \leq s-1$, $0\leq j \leq ms+2)$ 
          is a $K$-basis of $\Ima \widehat{\partial}^{ms+2}$ for $m$ odd, $s$ odd, and {\rm char}\,$K\neq 2$. 
    \item $\delta_{i,j}^{ms+2}+\delta_{i-1,j}^{ms+2}$ $(0\leq i \leq s-2$, $0\leq j \leq ms+2)$ 
          is a $K$-basis of $\Ima \widehat{\partial}^{ms+2}$ for  $m$ even, $s$ even, or {\rm char}\,$K= 2$. 
    \item $\Ima \widehat{\partial}^{ms+r+1}=0$ for $r\neq 0,1$. 
   \end{enumerate} 
 \end{lem}
\noindent
As an immediate consequence of the above lemma, 
we get the dimension of $\Ima\widehat{\partial}^{n+1}$ for $n\geq 0$: 
 \begin{cor}
  Let $n=ms+r$ for integers $m\geq $ and $0\leq r \leq s-1$. 
  Then the dimension of $\Ima\widehat{\partial}^{n+1}$ is as follows{\rm :} 
  \begin{align*}
   &\dim_{K}\Ima\widehat{\partial}^{ms+r+1}\\
   &=\left\{
     \begin{array}{ll}
      s(ms+1)             &\quad \text{if}\ s\ \text{odd},\ m\ \text{odd},\ {\rm char}\, K\neq2\ \text{and}\ r=0, \\
      (s-1)(ms+1)         &\quad \text{if}\ s\ \text{even}\ \text{and}\ r=0,\ \text{if}\ m\ \text{even\ and}\ r=0,\text{or}\\
                          &\quad \quad \quad \text{if}\ {\rm char}\,K=2\ \text{and}\ r=0, \\
      s(ms+3)             &\quad \text{if}\ s\ \text{odd},\ m\ \text{odd},\ {\rm char}\,K\neq2\ \text{and}\ r=1, \\
      (s-1)(ms+3)         &\quad \text{if}\ s\ \text{even\ and\ } r=1,\ \text{if}\ m\ \text{even\ and}\ r=1, \text{or}\\
                          &\quad \quad \quad \text{if}\ {\rm char\,}K=2\ \text{and}\ r=1, \\
      0                   &\quad otherwise. 
     \end{array}
    \right.
   \end{align*}
 \end{cor}
\subsection{A basis of $\Ker \widehat{\partial}^{n+1}$}
Now, by using Lemma 3.4, 
we have the following lemma. 
The proof follows from easy computations. 
 \begin{lem}
  Let $n=ms+r$ for integers $m\geq 0$ and $0\leq r \leq s-1$. 
  The following elements give a $K$-basis 
  of the subspace $\Ker \widehat{\partial}^{n+1}$ of 
  $\widehat{Q}^{n}${\rm:} 
   \begin{enumerate}[\rm(a)]
    \item $\Ker\widehat{\partial}^{ms+1}=0$ for $m$ odd, $s$ odd, and {\rm char}\,$K\neq 2$. 
    \item $\sum_{i=0}^{s-1} \alpha_{i,j}^{ms}$ $(0\leq j\leq ms)$ 
          is a $K$-basis of $\Ker\widehat{\partial}^{ms+1}$ for $m$ even, $s$ even, or {\rm char}\,$K= 2$. 
    \item $\beta_{i,j+1}^{ms+1}+\gamma_{i,j}^{ms+1}$ $(0\leq i \leq s-1$, $0\leq j \leq ms)$ 
          is a $K$-basis of $\Ker\widehat{\partial}^{ms+2}$ for $m$ odd, $s$ odd, and {\rm char}\,$K\neq 2$. 
    \item $\sum_{i=0}^{s-1}\beta_{i,j}^{ms+1}$ $(0\leq j \leq ms+1)$, 
          $\beta_{i,j+1}^{ms+1}+\gamma_{i,j}^{ms+1}$ $0\leq i \leq s-2$, $0\leq j \leq ms)$, 
          $\sum_{i=0}^{s-1}\gamma_{i,j}^{ms+1}$ $(0\leq j \leq ms+1)$ 
          is a $K$-basis of $\Ker\widehat{\partial}^{ms+2}$ for $m$ even, $s$ even, or {\rm char}\,$K= 2$. 
    \item $\Ker \widehat{\partial}^{ms+3}=\widehat{Q}^{ms+2}$. 
    \item $\Ker \widehat{\partial}^{ms+r+1}=0$ for $r\neq 0,1,2$. 
   \end{enumerate}
 \end{lem}
\noindent
As an immediate consequence, 
we get the dimension of $\Ker\widehat{\partial}^{n+1}$ for $n\geq 1$: 
 \begin{cor}
  Let $n=ms+r$ for integers $m\geq 0$ and $0\leq r \leq s-1$. 
  Then the dimension of $\Ker\widehat{\partial}^{n+1}$ is as follows{\rm :}
  \begin{align*}
   &\dim_{K}\Ker\widehat{\partial}^{ms+r+1}\\
   &=\left\{
     \begin{array}{ll}
      0                   &\quad \text{if}\ s\ \text{odd},\ m\ \text{odd},\ {\rm char}\, K\neq2\ \text{and}\ r=0, \\
      ms+1                &\quad \text{if}\ s\ \text{even}\ \text{and}\ r=0,\ \text{if}\ m\ \text{even\ and}\ r=0, \text{or}\\
                          &\quad \quad \quad \text{if}\ {\rm char}\,K=2\ \text{and}\ r=0, \\
      s(ms+1)             &\quad \text{if}\ s\ \text{odd},\ m\ \text{odd},\ {\rm char}\,K\neq2\ \text{and}\ r=1, \\
      (s+1)(ms+1)+2       &\quad \text{if}\ s\ \text{even\ and\ } r=1,\ \text{if}\ m\ \text{even\ and}\ r=1, \text{or}\\
                          &\quad \quad \quad \text{if}\ {\rm char\,}K=2\ \text{and}\ r=1, \\
      s(ms+3)             &\quad \text{if}\ r=2, \\
      0                   &\quad otherwise. 
     \end{array}
   \right.
   \end{align*}
 \end{cor}
\subsection{Calculation of the Hochschild cohomology groups $\HH^{n}(\Lambda_{s})$}
Now, by Lemmas 3.4 and 3.6, 
we have a $K$-basis of the Hochschild cohomology group $\HH^{n}(\Lambda_{s})$ for $n\geq 0$. 
 \begin{prop}
  Let $n=ms+r$ for integers $m\geq 0$ and $0\leq r \leq s-1$. 
  Then the following elements give a $K$-basis of $\HH^{n}(\Lambda_{s})$ of $\Lambda_{s}${\rm :}
   \begin{enumerate}[\rm(a)]
    \item $\HH^{ms}(\Lambda_{s})=0$ for $s$ odd, $m$ odd, and {\rm char}\,$K\neq2$. 
    \item $\sum_{i=0}^{s-1}\alpha_{i,j}^{ms}$ $(0\leq j\leq ms)$ 
          is a $K$-basis of $\HH^{ms}(\Lambda_{s})$ for $s$ even, $m$ even, or {\rm char}\,$K=2$. 
    \item $\HH^{ms+1}(\Lambda_{s})=0$ for $s$ odd, $m$ odd, and {\rm char }\,$K\neq2$. 
    \item $\sum_{i=0}^{s-1}\beta_{i,0}^{ms+1}$, $\sum_{i=0}^{s-1}\gamma_{i,j}^{ms+1}$ $(0\leq j \leq ms+1)$, $\beta_{s-1,j+1}^{ms+1}+\gamma_{s-1,j}^{ms+1}$ $(0\leq j\leq ms)$ 
          is a $K$-basis of $\HH^{ms+1}(\Lambda_{s})$ for $s$ even, $m$ even, or {\rm char}\,$K=2$. 
    \item $\HH^{ms+2}(\Lambda_{s})=0$ for $s$ odd, $m$ odd and {\rm char } $K\neq2$. 
    \item $\delta_{s-1,j}^{ms+2}$ $(0\leq j\leq ms+2)$ 
          is a $K$-basis of $\HH^{ms+2}(\Lambda_{s})$ for $s$ even, $m$ even, or {\rm char}\,$K=2$, 
    \item $\HH^{ms+r}(\Lambda_{s})=0$ for $r\neq0,1,2$. 
   \end{enumerate}
 \end{prop}
By Proposition 3.8, we have the following theorem. 
 \begin{thm}
  Let $n=ms+r$ for integers $m\geq 0$ and $0\leq r \leq s-1$. 
  Then, for $s\geq 3$, we have the dimension formula for $\HH^{n}(\Lambda_{s})${\rm :} 
  \begin{align*}
   &\dim_{K}\HH^{ms+r}(\Lambda_{s})\\
   &=\left\{
    \begin{array}{ll}
      ms+1          &\quad \text{if}\ s\ \text{even}\ \text{and}\ r=0,\ \text{if}\ m\ \text{even\ and}\ r=0, \text{or}\\
                    &\quad \quad \quad \text{if}\ {\rm char}\,K=2\ \text{and}\ r=0, \\
      2ms+4         &\quad \text{if}\ s\ \text{even\ and\ } r=1,\ \text{if}\ m\ \text{even\ and}\ r=1, \text{or}\\
                    &\quad \quad \quad \text{if}\ {\rm char}\,K=2\ \text{and}\ r=1, \\
      ms+3          &\quad \text{if}\ s\ \text{even\ and\ } r=2,\ \text{if}\ m\ \text{even\ and}\ r=2, \text{or}\\
                    &\quad \quad \quad \text{if}\ {\rm char}\,K=2\ \text{and}\ r=2, \\
      0             &\quad otherwise. 
    \end{array}
   \right.
  \end{align*}
 \end{thm}
\section{The Hochschild cohomology ring modulo nilpotence 
         $\HH^{\ast}(\Lambda_{s})/\mathcal{N}_{s}$}
Throughout this section, 
we keep the notation from Sections 2 and 3. 
Recall that the Hochschild cohomology ring 
of the algebra $\Lambda_{s}$ is defined to be the graded ring 
$$
\HH^{\ast}(\Lambda_{s}):=\Ext_{\Lambda_{s}^{e}}^{\ast}(\Lambda_{s},\Lambda_{s})=\bigoplus_{t\geq 0} \Ext_{\Lambda_{s}^{e}}^{t}(\Lambda_{s},\Lambda_{s})
$$
with the Yoneda product. 
Denote $\mathcal{N}_{\Lambda_{s}}$ by the ideal generated 
by all homogeneous nilpotent elements 
in $\HH^{\ast}(\Lambda_{s})$. 
Then the quotient algebra $\HH^{\ast}(\Lambda_{s})/\mathcal{N}_{\Lambda_{s}}$ 
is called the Hochschild cohomology ring modulo nilpotence of $\Lambda_{s}$. 
Note that $\HH^{\ast}(\Lambda_{s})/\mathcal{N}_{\Lambda_{s}}$ 
is a commutative graded algebra (see \cite{SnSo}). 
Our purpose of this section is to find generators and relations 
of $\HH^{\ast}(\Lambda_{s})/\mathcal{N}_{\Lambda_{s}}$ for $s\geq 3$. 
For simplicity, 
we denote the graded subalgebras $\bigoplus_{t\geq 0}\HH^{st}(\Lambda_{s})$ 
of $\HH^{\ast}(\Lambda_{s})$ by $\HH^{s\ast}(\Lambda_{s})$ 
and $\bigoplus_{t\geq 0}\HH^{2st}(\Lambda_{s})$ by $\HH^{2s\ast}(\Lambda_{s})$. 
Also, we denote the Yoneda product in $\HH^{\ast}(\Lambda_{s})$ by $\times$. 
Note that, by Lemma 3.3, 
$\Ima\widehat{\partial}^{st}=0$ and so 
$\HH^{st}(\Lambda_{s})=\Ker \widehat{\partial}^{st+1}$ 
for $s\geq 3$ and $t\geq 0$. 
 \begin{thm}
  For $s\geq 3$, 
  there are the following isomorphisms of commutative graded algebras{\rm :}
   \begin{enumerate}[\rm (i)]
    \item If $s$ is odd and {\rm char}\,$K\neq 2$, then 
           \begin{align*}
           &\HH^{\ast}(\Lambda_{s})/\mathcal{N}_{\Lambda_{s}}\cong \HH^{2s\ast}(\Lambda_{s})\\
           &\cong K[z_{0},\ldots,z_{2s}]/\langle z_{k}z_{l}-z_{q}z_{r}\ |\ k+l=q+r,\,0\leq k,l,q,r \leq 2s\rangle, 
           \end{align*}
          where $z_{0},\ldots,z_{2s}$ are in degree $2s$. 
    \item If $s$ is even or {\rm char}\,$K=2$, then 
           \begin{align*}
            &\HH^{\ast}(\Lambda_{s})/\mathcal{N}_{\Lambda_{s}}\cong \HH^{s\ast}(\Lambda_{s})\\
            &\cong K[z_{0},\ldots,z_{s}]/\langle z_{k}z_{l}-z_{q}z_{r}\ |\ k+l=q+r,\,0\leq k,l,q,r \leq s\rangle, 
            \end{align*}
           where $z_{0},\ldots,z_{s}$ are in degree $s$. 
   \end{enumerate}
  Therefore, $\HH^{\ast}(\Lambda_{s})/\mathcal{N}_{\Lambda_{s}}$ 
  is finitely generated as an algebra. 
 \end{thm}
 \begin{proof}
  We prove (i) only. The proof of (ii) is similar. 
  First, we construct the second isomorphism in the statement. 
  For $0 \leq u \leq 2s$, 
  we set $z_{u}:=\sum_{i=0}^{s-1}\alpha_{i,u}^{2s}$: 
  $Q^{2s}\rightarrow \Lambda_{s}$, 
  where $\alpha_{i,u}^{2s}$ is the map defined in Section 3.1. 
  For $0 \leq u \leq 2s$ and $v\geq 0$, 
  we define a $\Lambda_{s}^{e}$-module homomorphism 
  $\theta_{u}^{v}$: 
  $Q^{2s+v}\rightarrow Q^{v}$ by 
   $$
    b_{k,l}^{2s+v}\longmapsto 
     \left\{
      \begin{array}{ll}
      b_{k, w}^{v} &\quad \text{if }l=u+w\,\text{for some integer }w\text{ with }0\leq w \le v,\\
      0            &\quad \text{otherwise}
      \end{array}
     \right.
   $$
  for $0\leq k\leq s-1$ and $0\leq l\leq 2s+v$. 
  Then $z_{u}=\partial^{0}\theta_{u}^{0}$ and 
  $\theta_{u}^{v}\partial^{2s+v+1}=\partial^{v+1}\theta_{u}^{v+1}$ 
  hold for $0\leq u \leq 2s$ and $v\geq 0$, 
  and so $\theta_{u}^{v}$ is a lifting of $z_{u}$ $(0\leq u\leq 2s)$. 
  Also, it follows that, for any integers $0\leq u_{1},u_{2}\leq 2s$, 
  the composite $z_{u_{2}}\theta_{u_{1}}^{2s}$: $Q^{4s}\rightarrow \Lambda_{s}$ 
  is given by 
  $$
  b_{k,l}^{4s} \longmapsto 
  \left\{
   \begin{array}{ll}
   e_{k} &\quad \text{if }l=u_{1}+u_{2}, \\
   0     &\quad \text{otherwise}
   \end{array}
   \right.
  $$
  for $0\leq k\leq s-1$ and $0\leq l\leq 4s$. 
  Hence we have $z_{u_{2}}\theta_{u_{1}}^{2s}=\sum_{i=0}^{s-1}\alpha_{i,u_{1}+u_{2}}^{4s}$, 
  and this equals the Yoneda product $z_{u_{1}}\times z_{u_{2}}\in \HH^{4s}(\Lambda_{s})$. 
  
  Now, let $t\geq 2$ be a positive integer, and 
  let $u_{1},\ldots,u_{t}$ be integers with $0\leq u_{1},\ldots,u_{t}\leq 2s$. 
  Then it is proved by induction on $t$ that the product 
  $z_{u_{1}}\times \cdots \times z_{u_{t}}$ 
  equals the map 
  $$
  Q^{2st}\rightarrow \Lambda_{s};\,b_{k,l}^{2st}\longmapsto 
  \left\{
    \begin{array}{ll}
     e_{k} &\quad \text{if }l=\sum_{p=1}^{t} u_{p}, \\
     0     &\quad \text{otherwise}
   \end{array}
  \right.
  $$
  for $0\leq k\leq s-1$ and $0\leq l\leq 2st$, 
  which equals the map $\sum_{i=0}^{s-1}\alpha_{i,\sum_{p=1}^{t}u_{p}}^{2st}$. 
  Therefore, by Proposition 3.8 (b), 
  we see that $\HH^{2s\ast}(\Lambda_{s})$ 
  is generated by $z_{0},\ldots,z_{2s}\in \HH^{2s}(\Lambda_{s})$. 
  
  Let $t\geq 2$ be an integer, 
  and let $z_{u_{1}}\times \cdots \times z_{u_{t}}$ and 
  $z_{u'_{1}}\times \cdots \times z_{u'_{t}}$
  be any products  in $\HH^{2st}(\Lambda_{s})$ 
  for $0\leq u_{p},u'_{p} \leq 2s$ $(1\leq p \leq t)$. 
  Then, since $z_{u_{1}}\times \cdots \times z_{u_{t}}=\sum_{i=0}^{s-1}\alpha_{i,\sum_{p=1}^{t}u_{p}}^{2st}$
  and $z_{u'_{1}}\times \cdots \times z_{u'_{t}}=\sum_{i=0}^{s-1}\alpha_{i,\sum_{p=1}^{t}u'_{p}}^{2st}$, 
  it follows that 
  $z_{u_{1}}\times \cdots \times z_{u_{t}}=z_{u'_{1}}\times \cdots \times z_{u'_{t}}$ 
  if and only if $\sum_{p=1}^{t}u_{p}=\sum_{p=1}^{t}u'_{p}$. 
  This means that the relations $z_{k}z_{l}-z_{q}z_{r}=0$ 
  for every $0\leq k,l,q,r \leq 2s$ with $k+l=q+r$ 
  are enough to give the second isomorphism. 
  
  Now, using the second isomorphism, 
  we easily see that all elements in $\HH^{2s\ast}(\Lambda_{s})$ are not nilpotent. 
  Furthermore, for $t\geq 0$ and $r=1,\ldots ,s-1$, 
  the image of all basis elements of $\HH^{2st+r}(\Lambda_{s})$ described in 
  Proposition 3.8 are in $\rad \Lambda_{s}$, 
  so that by \cite[Proposition 4.4]{SnSo}, 
  $\HH^{2st+r}(\Lambda_{s})$ is contained in $\mathcal{N}_{\Lambda_{s}}$. 
  Hence we have the first isomorphism. 
 \end{proof}
We conclude this paper with the following remarks. 
\begin{rem}
Let $E(\Lambda_{s})=\bigoplus_{i\geq 0}\Ext_{\Lambda_{s}}^{i}(\Lambda_{s}/\rad \Lambda_{s},\Lambda_{s}/\rad \Lambda_{s})$ 
      be the Ext algebra of $\Lambda_{s}$, 
      and let $Z_{gr}(E(\Lambda_{s}))$ be the graded center of $E(\Lambda_{s})$ 
      (see \cite{BGSS}, for example). 
      Denote by $\mathcal{N}'_{\Lambda_{s}}$ the ideal of $Z_{gr}(E(\Lambda_{s}))$ 
      generated by all homogeneous nilpotent elements. 
      Since $\Lambda_{s}$ is a Koszul algebra by Proposition 2.2, 
      it follows by \cite{BGSS} that 
      $Z_{gr}(E(\Lambda_{s}))/\mathcal{N}'_{\Lambda_{s}}\cong \HH^{\ast}(\Lambda_{s})/\mathcal{N}_{\Lambda_{s}}$ 
      as graded rings. 
      Therefore, we have the same presentation of 
      $Z_{gr}(E(\Lambda_{s}))/\mathcal{N}'_{\Lambda_{s}}$ 
      by generators and relations as that in Theorem 4.1. 
\end{rem}
\begin{rem}
      In \cite{F}, Furuya has discussed the Hochschild cohomology of some self-injective 
      special biserial algebra $A_{T}$ for $T\geq 0$, 
      and in particular he has given a presentation of $\HH^{\ast}(A_{T})/\mathcal{N}_{A_{T}}$ 
      by generators and relations in the case $T=0$. 
      We easily see that the algebra $\Lambda_{4}$ is isomorphic to $A_{0}$. 
      By setting $s=4$ in Theorem 4.1, 
      our presentation actually coincides with that in \cite[Theorem 4.1]{F}. 
\end{rem}
\section*{Acknowledgments}
The author is grateful to Professor Katsunori Sanada 
and Professor Takahiko Furuya for many helpful comments and suggestions 
on improving the clarity of the paper. 


\begin{thebibliography}{HD}
   \bibitem[ASS]{ASS}
   I. Assem, D. Simson and A. Skowro\'nski, 
   {\it Elements of the representation theory of associative algebras}, 
   Vol. 1. Techniques of representation theory. 
   London Mathematical Society Student Texts, 65. 
   Cambridge University Press, Cambridge, 2006. 
  
  \bibitem[BGSS]{BGSS}
   R.-O. Buchweitz, E. L. Green, N. Snashall and \O. Solberg, 
   Multiplicative structures for Koszul algebras, 
   Quart. J. Math. 59 (2008), no. 4, 441-454. 

  \bibitem[E]{E}
  L. Evens, 
  The cohomology ring of a finite group, 
  Trans. Amer. Math. Soc. 101 (1961), 224-239. 
  
  \bibitem[ES]{ES}
   K. Erdmann and S. Schroll, 
   On the Hochschild cohomology of tame Hecke algebras, 
   Arch. Math. (Basel) 94 (2010), no. 2, 117-127. 
  
  \bibitem[F]{F}
   T. Furuya, 
   Hochschild cohomology for a class 
   of some self-injective special biserial algebras of rank four, 
   arXiv:1403.6375. 
   
  \bibitem[FO]{FO}
   T. Furuya and D. Obara, 
   Hochschild cohomology of a class of 
   weakly symmetric algebras with radical cube zero, 
   SUT J. Math. 48 (2012), no. 2, 117-143. 
   
  \bibitem[GHMS]{GHMS}
   E. L. Green, G. Hartmann, E. N. Marcos and \O. Solberg, 
   Resolutions over Koszul algebras, 
   Arch. Math. (Basel) 85 (2005), 118-127. 
   
  \bibitem[GSS1]{GSS1}
  E. L. Green, N. Snashall and \O. Solberg, 
  The Hochschild cohomology ring of a selfinjective algebra 
  of finite representation type, 
  Proc. Amer. Math. Soc. 131 (2003), no. 11, 3387-3393 (electronic). 
  
  \bibitem[GSS2]{GSS2}
  E.L. Green, N. Snashall and \O. Solberg, 
  The Hochschild cohomology ring modulo nilpotence of a monomial algebra, 
  J. Algebra Appl. 5 (2006), no. 2, 153-192.
  
  \bibitem[GSZ]{GSZ}
   E.L. Green, \O. Solberg and D. Zacharia, 
   Minimal projective resolutions, 
   Trans. Amer. Math. Soc. 353 (2001), 2915-2939. 
   
  \bibitem[S]{S}
   N. Snashall, 
   Support varieties and the Hochschild cohomology ring modulo nilpotence, 
   Proceedings of the 41st Symposium on Ring Theory and Representation Theory, 
   68-82, Ed. H. Fujita, Tsukuba, 2009.
   
  \bibitem[ScSn]{ScSn}
   S. Schroll and N. Snashall, 
   Hochschild cohomology and support varieties for tame Hecke algebras, 
   Quart. J. Math. 62 (2011), no. 4, 1017-1029. 
   
  \bibitem[SnSo]{SnSo}
   N. Snashall and \O. Solberg, 
   Support varieties and Hochschild cohomology rings, 
   Proc. London Math. 81 (2004), 705-732. 
   
  \bibitem[ST]{ST}
   N. Snashall and R. Taillefer, 
   The Hochschild cohomology ring of a class of special biserial algebras, 
   J. Algebra Appl. 9 (2010), no. 1, 73-122. 
   
  \bibitem[X]{X}
   F. Xu, 
   Hochschild and ordinary cohomology rings of small categories, 
   Adv. Math. 219 (2008), no. 6, 1872-1893. 
  
  \bibitem[XH]{XH}
   Y. Xu and Y. Han, 
   Hochschild {\rm (}co{\rm )}homology of exterior algebras, 
   Comm. Algebra 35 (2007), no. 1, 115-131. 

  \bibitem[XZ]{XZ}
   Y. Xu and C. Zhang, 
   More counterexamples to Happel's question and Snashall-Solberg's conjecture, 
   arXiv:1109.3956. 
\end{thebibliography}
\end{document}